\documentclass[12pt,a4paper,twoside,reqno]{amsart}

\usepackage[all,cmtip]{xy}
\usepackage{amsmath,amscd}
\usepackage{amssymb,amsfonts}
\usepackage[driver=pdftex,margin=3cm,heightrounded=true,centering]{geometry}
\usepackage{JK}
\usepackage[colorlinks=true,linkcolor=blue,citecolor=blue]{hyperref}

\newtheorem{thm}{Theorem}[section]
\newtheorem{lem}{Lemma}[section]
\newtheorem{notation}[thm]{Notation}
\usepackage{graphicx}

\begin{document}
\title[Tate tame symbol and the joint torsion of commuting operators]{Tate tame symbol and the joint torsion of commuting operators}

\author{Jens Kaad}
\address{International School of Advanced Studies (SISSA),
Via Bonomea 265,
34136 Trieste,
Italy}
\email{jenskaad@hotmail.com}

\author{Ryszard Nest}
\address{Department of Mathematical Sciences,
Universitetsparken 5,
DK-2100 Copenhagen {\O},
Denmark}
\email{rnest@math.ku.dk}
%
%
%
\thanks{The second author is partially supported by the Danish National Research Foundation (DNRF) through the Centre for Symmetry and Deformation.}

\begin{abstract}
We investigate determinants of Koszul complexes of holomorphic functions of a commuting tuple of bounded operators acting on a Hilbert space. Our main result shows that the analytic joint torsion, which compares two such determinants, can be computed by a local formula which involves a tame symbol of the involved holomorphic functions. As an application we are able to extend the classical tame symbol of meromorphic functions on a Riemann surface to the more involved setting of transversal functions on a complex analytic curve. This follows by spelling out our main result in the case of Toeplitz operators acting on the Hardy space over the polydisc.
\end{abstract}

\maketitle
\tableofcontents
\section{Introduction}

The main theme of this series of papers (began with \cite{KaNe12} and \cite{KaNe14}) is the study of the algebraic structure of a Hilbert space $\sH$ as a module over an algebra of holomorphic functions of a finite family of bounded commuting operators on $\sH$. 

The simplest example is the case of the Toeplitz operator $T_z$ acting on the Hardy space over the open unit disc, $H^2({\mathbb D})$. Here $z$ denotes the complex coordinate in $\cc$. The holomorphic functional calculus of bounded operators furnishes the Hilbert space $H^2({\mathbb D})$ with the structure of a module over the ring $\C{O}(\ov{\B D})$  of germs of functions holomorphic in a neighbourhood of the closed unit disc, and analysis of Toeplitz operators is closely related to the algebraic properties of this module structure. 

The prototype of results of the type we are interested in are as follows.

\medskip

\noindent{\bf Theorem (Fritz Noether)}
\noindent{\it
Given a holomorphic function $f$ on a neighbourhood of $\mathbb D$ and invertible on its boundary $\partial {\mathbb D}$, the Toeplitz operator $T_f$ is Fredholm and its index is given by (minus) the winding number of $f$,
\begin{equation}\label{eq:Noether}
\T{Ind}(T_f)=-\frac{1}{2\pi i}\int_{\mathbb{T}}f^{-1}df=-\sum_{|\la |<1} \T{res}_{\la }( f^{-1}df).
\end{equation}}

A way of interpreting this result is to notice that the left hand side of this equality is an {\it analytic object}, the $L^2$-Euler characteristic of the complex
\[
\xymatrix{0\ar[r]&\ar[r]^{T_f}\sH&\sH\ar[r]&0,
}
\]
while the right hand side has an algebraic $K$-theory interpretation. In fact, let $\sK$ denote the field of fractions of $\C{O}(\ov{\B D})$. Then an $f\in \C{O}(\ov {\B D})$ whose restriction to the unit circle is invertible determines an element $[f]\in K_1^{\T{alg}}(\sK)$ and, in this situation, the right hand side of equation (\ref{eq:Noether}) becomes the residue of $\log f$, which should be thought of as the result of the composition of the Chern character map from algebraic K-theory of $\sK$ to the {\it Deligne cohomology $H^*_{\C D} (\mathbb{D}\sem S)$} (where $S \su \B D$ is a finite set)
\[
K_1^{\T{alg}}({\sK})\ni [f]\rightarrow [\log f]\in H^1_{\C D}(\mathbb{D}\sem S,\zz(1)),
\]
with the residue map (see for example \cite{Dcoh})\label{ref. Dcoh}.

While this particular case does not really justify the reference to algebraic K-theory, the corresponding computations in the case of $K_2^{\T{alg}}$ become much more demanding. In their work on algebraic K-theory of the quotient of the algebra of bounded operators by the ideal of trace class operators, Carey and Pincus introduced an analytic object, the joint torsion $\T{JT}(A,B)$ of a pair of commuting Fredholm operators $A$ and $B$. In the particular case of a pair of Toeplitz operators $T_f,\ T_g$ on the Hardy space $H^2(\mathbb D)$, they proved the following theorem:

\medskip

\noindent{\bf Theorem (Carey, Pincus)}
\noindent{\it Suppose that both $T_f$ and $T_g$ are Fredholm. Then
\begin{equation}\label{eq:C-P formula}
\T{JT}(T_f,T_g)=\prod_{\la\in \mathbb D} (-1)^{m_\la(g) \cd m_\la(f)} \lim_{z \to \la} \frac{g(z)^{m_\la(f)}}{f(z)^{m_\la(g)}} \in \cc^*
\end{equation}
where $m_\la(f)$ is the multiplicity of a zero of $f$ at $\la$.} 

\medskip

The left hand side of this equality is a certain {\it analytic} invariant of the pair of operators $(T_f,T_g)$, while  the right hand side is again best understood via evaluating the Chern character, this time on the Steinberg symbol $[f,g]\in K_2^{alg}(\sK)$. The Chern character provides a class 
\[
[\log f]\cup[\log g ]\in H^2_{\C D}(\mathbb{D} \sem S,\zz(2))
\]
and the right hand side of the equation (\ref{eq:C-P formula}) becomes the exponential of the residue of $[\log f d(\log g)]$ (see f. ex. \cite{ES1}). This exponential is also known as the Tate tame symbol of the class $[f,g]\in K^{\T{alg}}_2 (\sK)$ (see \cite{Tate}, \cite{DeligneTS} ).

\medskip 

The subject of this paper is the generalization of these results to the case of a commuting n-tuple $A = (A_1,\ldots,A_n)$  of bounded operators on a Hilbert space $\sH$. The basic idea is to replace the single Toeplitz operator $T_z$ by the Koszul complex $K(A,\sH)$ of our arbitrary n-tuple, see Subsection \ref{ss:Koszul}. The generalization of the closed unit disc in the Toeplitz case becomes the Taylor  spectrum $\spa \su \cc^n$, the set of $\la\in\cc^n$, for which $K(A-\la ,\sH)$, the Koszul complex of $A-\la = (A_1-\la_1,\ldots,A_n-\la_n)$, is not contractible. The  notions of essential (resp. Fredholm) spectrum correspond to the values of $\la$ for which $K(A-\la ,\sH)$ has infinite- (resp. non-trivial finite-) dimensional cohomology. In the case that $\la$ belong to the Fredholm spectrum of $A$, the index $\T{Ind}(A-\la)$ stands for the Euler characteristic of the cohomology of the Koszul complex $K(A-\la,\sH)$.

By a theorem of Taylor (\cite{Tay:ACS}), the standard holomorphic functional calculus extends to a multivariable holomorphic functional calculus of such an n-tuple $A$, i.e.

\medskip

\noindent {\it $\sH$ is a module over the ring $\mathcal{O}(\T{Sp}(A))$ of germs of functions holomorphic in a neighbourhood of $\spa$.} 

\medskip

\noindent Our goal is to study the algebraic structure of this module. 

\bigskip

The analogue of the Fritz Noether index theorem was dealt with in \cite{KaNe12}. Since it is needed to formulate the results of this paper, we recall the statement.

\medskip

\noindent{\bf Theorem} {\it  Let $f_1,\ldots,f_n\in\C{O}(\T{Sp}(A))$. The n-tuple 
\[
f(A)=(f_1(A_1,\ldots,A_n),\ldots,f_n(A_1,\ldots,A_n))
\]
is Fredholm iff the set $Z(f)$ of common zeroes of $f_i$'s within $\T{Sp}(A) $ has no intersection with the essential spectrum of $A$. If $f(A)$ is Fredholm, the following holds:
\begin{enumerate}
\item The set $Z(f)$ of common zeroes is finite;
\item \begin{equation}\label{eq:local index}
\T{Ind}(f(A))=\sum_{\la\in Z(f)} \T{Ind}_\la(f) \, , \q 
\T{Ind}_\la(f) = m_\la(f) \cd \T{Ind}(A-\la).
\end{equation}
\noindent Here $m_\la(f)$ are the multiplicities of the points $\la\in Z(f)$ given by the dimensions
\[
m_\la (f)=\T{dim}_\cc \big( \C O_\la/\inn{f}_\la \big) \ ,
\]
where $\inn{f}_\la$ is the ideal generated by $f$ in the stalk $\C O_\la$ of convergent power series near $\la$. 
\end{enumerate} 
}
\medskip

The joint torsion, the generalization of the torsion invariant of Carey and Pincus, is defined in the following context. First a bit of notation. A complex $\mathcal{C}$ of Hilbert spaces is Fredholm if it has finite dimensional cohomology and, if $\mathcal{C}$ is Fredholm, its determinant line is the one dimensional vector space
\[
|\C C|=\Lambda^{\T{top}}H^+(\C C) \otimes (\Lambda^{\T{top}}H^-(\C C))^\da \ ,
\]
\noindent (see Subsection \ref{ss:frecom} for more details).  

\medskip

\noindent{\bf Definition} {\it Let $\mathcal{C}$ be a complex of Hilbert spaces and $f,g:\mathcal{C}\rightarrow \mathcal{C}$ be two commuting morphisms of complexes. Suppose that the mapping cones $C_f$ and $C_g$ of $f$ and $g$ are both Fredholm. Then the long exact cohomology sequences of the mapping cone of $f$ acting on $C_g$ and of $g$ acting on $C_f$ provide two trivializations of the determinant line of $C_{f,g}$ and the joint torsion $\T{JT}(\C C; f,g)$ is the quotient of these two trivializations.} 
\medskip

\noindent See \cite[Section 3.3]{Kaa:JTS} and Subsection \ref{ss:torsio} for details. \footnote{\noindent {\it The choice of signs in this paper is not standard, but is dictated by Theorem \ref{th:analyticity}. }}.
\medskip

The notion of joint torsion for the $\mathcal{O}(\T{Sp}(A))$-module $\sH$ appears in the following situation. 
Let $h_1,\ldots,h_{n-1},f,g$ be holomorphic functions defined in a neighbourhood of  $\spa$ and suppose that the zero sets $Z(h) \cap Z(f)$ and $Z(h) \cap Z(g)$ do not intersect the essential spectrum of $A$. In this case the commuting tuples $(h_1(A),\ldots, h_{n-1}(A),f(A))$ and $(h_1(A),\ldots, h_{n-1}(A),g(A))$ are Fredholm and the joint torsion $\T{JT}(K(h,\sH);f,g) \in \cc^*$ is well defined. The main result is as follows (Theorem \ref{t:joytorgen}).

\begin{thm}\label{th,intro,global} Joint torsion is multiplicative, i.e. 
\[
\T{JT}\big( K(h,\sH) ; f,g \big)=\prod_{\la\in Z(h) \cap \big(Z(f) \cup Z(g)\big)}
c_\la(h;f,g)^{\T{Ind}(A - \la)} \ ,
\]
where the local terms $c_\la(h;f,g)$ only depend on the image of the functions $f,\ g,\ h_i,i=1,\ldots, n-1$ in the stalk of $\C O_\la$.
\end{thm}
The local terms in the product above are given by the following result (Theorem \ref{t:joystalk}).

\begin{thm}\label{t:joystalkint}
Let $U \su \cc^n$ be open with compact closure $\ov U$ and let $h_1,\ldots,h_{n-1},f,g \in \C O(\ov U)$. Suppose that $\la \in U$ satisfies
\[
Z(h) \cap \big( Z(f) \cup Z(g) \big) \su \{\la\}.
\]
Then the sequence of quotients
\[
c_\la(h^k;f^k,g^k) = \frac{\prod_{\mu \in U \cap Z(h^k) \cap Z(g^k)} f^k(\mu)^{m_\mu(h^k,g^k)}} 
{\prod_{\nu \in U \cap Z(h^k) \cap Z(f^k)} g^k(\nu)^{m_\nu(h^k,f^k)}}
\]
converges to $c_\la(h;f,g)$ for any sequences $\{h_i^k\}, \ \{f^k\}, \ \{g^k\}$ in $\C O(\ov U)$
which converge uniformly to $h_i,f,g$ and for which
\[
Z(h^k) \cap Z(f^k) \cap Z(g^k) = \emptyset \q \forall k \in \nn \ .
\]
\end{thm}

\begin{remark}
The assumptions of Theorem \ref{th,intro,global} can be weakened. In fact it is sufficient for the functions  $f$ and $g$ to be defined and holomorphic in a neighbourhood of $X=Z(h)\cap \T{Sp}( A)$. Using the methods of \cite{Mueller}, Taylor holomorphic calculus can be extended to define operators $f(A)$ and $g(A)$ acting on the Koszul complex $K(h,\sH)$ and the conclusions of the theorem will still hold. 
\end{remark}

\begin{remark}
One can also define the ``local joint torsion'' by localizing the Hilbert space at the prime ideal $\G p_\la$ of functions in $\C O(\T{Sp}(A))$ vanishing at $\la$. The resulting numbers are conjecturally the same as the local joint torsion $c_\la(h;f,g)^{\T{Ind}(A - \la)}$ as defined above.
\end{remark}

As an application of our main results we are able to extend the definition of the Tate tame symbol from Riemann surfaces to more general complex analytic curves. More precisely, we work with a complex analytic curve $(X,\C O_X)$ and a fixed point $x \in X$ such that there exists a local model $\big( Z(h), \C O/\inn{h}\big|_{Z(h)}\big)$ near $x$ which is determined by a holomorphic map
\[
h = (h_1,\ldots,h_{n-1}) : U \to \cc^{n-1} 
\]
where $U \su \cc^n$ is an open set. For any two holomorphic functions $f,g \in \C O_X(V)$ which are transversal to the curve near $x \in X$ we may then define the Tate tame symbol
\[
c_x(X;f,g) \in \cc^*
\] 
by applying the local description of Theorem \ref{t:joystalkint}. This invariant satisfies the properties which define a symbol in arithmetic, see \cite{Tat:SA}.

\begin{remark}
The conditions on the specific type of local model near $x \in X$ can be removed by working with more general resolutions than the Koszul complex. In particular, we will be able to study tame symbols of general complex analytic curves. We plan to carry out the details in a future publication.
\end{remark}
\medskip

The structure of the paper is a follows.

\medskip

\noindent Section \ref{s:Determinants, torsion and joint torsion} is devoted to basic definitions involving determinants and torsion of Fredholm complexes. The definition basic to this paper, that of joint torsion, is given in Subsection \ref{ss:torsio}. 

\noindent The notion underlying these constructions is a determinant functor on the triangulated category of Fredholm complexes but, since this more abstract context is not necessary to understand what follows, we have avoided this language.
\medskip

\noindent Section \ref{s:Joint torsion of commuting operators} is devoted to the generalities involving Koszul complexes of commuting families of bounded operators. After recalling the basic definitions (Subsection \ref{ss:Koszul}) we describe the localization procedure involved in the computation of local indices and state the local index theorem (see Subsection \ref{ss:local indices}).
\medskip

\noindent Section \ref{s:Multiplicative Lefschetz numbers} contains the basic technical computations involving the joint torsion. The main result of this section, Theorem \ref{t:nonsin}, gives a formula for joint torsion in the case when the (n+1)-tuple $(h,f,g)$ has no common zeroes in $\T{Sp}(A)$. The proof is based on the observation that in this case the joint torsion is given by the quotient of two determinants, both of which can be computed explicitly\footnote{The fact that the joint torsion of two Fredholm operators $(A_1,A_2)$ whose Koszul complex is contractible can be given  as a quotient of two determinants holds in fact in a more general context. It is sufficient to assume that $A_1$ and $A_2$ commute up to a trace class operator and that, moreover, one can construct an acyclic complex \begin{footnotesize}$\xymatrix{\sH\ar[r]^{d_1}&\sH\op\sH \ar[r]^{d_2}&\sH}$\end{footnotesize}, where the boundary maps are trace class perturbations of $(A_1,A_2)^t$ (resp. 
$(A_2,-A_1)$), see \cite{Migler}}.
\medskip

\noindent Section \ref{s:Localization of the joint torsion} contains the proofs of the main theorems listed above. This section relies heavily on the continuity properties of the joint torsion as investigated by the authors in \cite{KaNe14}.
\medskip

\noindent Section \ref{s:tamsymcom} contains the application of our results to the setting of complex analytic curves.

\section{Determinants, torsion and joint torsion}\label{s:Determinants, torsion and joint torsion}


\noindent \emph{Throughout this section $\ff$ will be a fixed field of characteristic zero.}

\subsection{Determinants of vector spaces}\label{s:detfin}

\subsubsection{Picard category of graded lines}

\begin{dfn}\label{dfn:Picard}
$\G{L}$ denotes the category of $\zz$-graded lines over $\ff$. The objects of $\G{L}$ are thus pairs $(V,n)$ where $V$ is a one-dimensional vector space over $\ff$ and $n \in \zz$. The set of morphisms $\T{Mor}\big( (V,n), (W,m) \big)$ is  the set of isomorphisms $V \to W$ when $n = m$ and empty when $n \neq m$.
\end{dfn}

The category $\G{L}$ becomes a {\it Picard category} when equipped with the bifunctor 
\[
\ot : \big( (V,n), (W,m) \big) \mapsto (V \ot W,n+m)
\]
which satisfies the obvious associativity constraint and the commutativity constraint 
\[
\psi_{(V,n),(W,m)} : (V,n) \ot (W,m) \to (W,m) \ot (V,n) \q \xi \ot \eta \mapsto (-1)^{n \cd m} \eta \ot \xi.
\]
$\da : \G{L} \to \G{L}$ denotes the covariant functor: 
\[
\da(V,n):= (V^*,-n)\mbox{ on objects and $\da(\al) :=  (\al^{-1})^*$ on morphisms,}
\]
where the superscript ${(\cdot )}^{*}$ denotes the linear dual (resp. transpose) of a vector space (resp. linear transformation).

Below we will often use the superscript ${}^\da$ to denote the image under $\da$. 
Together with the natural isomorphisms 
\[
\begin{split}
& c_{(V,n),(W,m)} : (V,n)^\da \ot (W,m)^\da \to \big( (V,n) \ot (W,m)\big)^\da \\
& \big(c_{(V,n),(W,m)}(\la\ot \mu)\big)(\xi \ot \eta) := \la(\xi) \cd \mu(\eta) \cd (-1)^{n\cd m},
\end{split}
\]
the covariant functor $\da$ becomes a monoidal functor. Furthermore, for any graded line $(V,n)$, the image $(V,n)^\da = (V^*,-n)$ together with the isomorphism $\ep_{(V,n)} : V \ot V^* \to \ff$, $\ep_{(V,n)} : \xi \ot \la \mapsto \la(\xi)$ is a fixed right inverse. Here the ground field $(\ff,0)$ and the obvious isomorphisms $V \ot \ff \cong V \cong \ff \ot V$ play the role of a fixed unit.
%
%

\subsubsection{Graded vector spaces}
\begin{dfn}$\mbox{}$
$\G{V}$ denotes the abelian category of \emph{finite} dimensional vector spaces over $\ff$ and  $\G{V}_{\T{iso}}$ denotes the  subcategory of $\G{V}$ with the same objects as $\G{V}$ and where
\[
\T{Mor}_{\G{V}_{\T{iso}}}(V,W)=\{A\in \T{Mor}_{\G{V}}(V,W)\mid A\mbox{ is invertible}\}
\]
$|\cdot|$ will denote the determinant functor given by
\[
\G{V}_{\T{iso}}\ni V\rightarrow (\Lambda^{\T{top}}V,\T{dim}V)\in \G{L}
\]
on objects and by $f \rightarrow \Lambda^{\T{top}}(f )$ on morphisms (invertible linear transformations), where $\La(V)$ denotes the exterior algebra over $V$.
\end{dfn}
\begin{remark}
Note that it is a determinant functor as defined for example in \cite[Definition 2.3]{Bre:DTC} (see also \cite{Del:DC} or \cite[Definition 1.4]{Knu:DEB}).  
\end{remark}
The basic property of the determinant functor is the following simple observation. Given a short exact sequence of finite dimensional vector spaces
\[
\xymatrix{
\De : 0\ar[r]& V\ar[r]^\iota& W\ar[r]^\pi& Z\ar[r]& 0,
}
\]
there exists an associated canonical isomorphism
\begin{equation}
\label{iso}
|\De|:|W|\rightarrow |V|\otimes |Z|
\end{equation}
given as follows. Let $v_1,\ldots,v_{\T{dim} V}$ be a linear basis of the image of $V$ in $W$ and $w_1,\ldots ,w_{\T{dim}Z}$ its completion to a linear basis of $W$. Then
\[
\begin{split}
& |\De|(v_1\wedge\ldots\wedge v_{\T{dim}V}\wedge w_1\wedge\ldots\wedge w_{\T{dim}Z} ) \\
& \qq = (-1)^{\T{dim} V \cd \T{dim} Z} \big( \iota^{-1}(v_1)\wedge\ldots\wedge \iota^{-1}(v_{\T{dim}V}) \big) \otimes \big(\pi(w_1)\wedge\ldots\wedge \pi(w_{\T{dim}Z} )\big).
\end{split}
\]
It is straightforward to check that $|\De|$ is independent of the choices made. Remark that the extra sign $(-1)^{\T{dim} V \cd \T{dim} Z}$ is non-standard.
\medskip

This determinant functor extends to the  category of $\zz/2\zz$-graded finite dimensional vector spaces. First some notation.
\begin{notation}$\mbox{}$
\begin{enumerate}
\item $\G{V}^{\zz/2\zz}$ denotes the category of finite dimensional $\zz/2\zz$-graded vector spaces with objects $V = V^+ \op V^-$ and morphisms $\alpha = \ma{cc}{\al^+ & 0 \\ 0 & \al^-} : V \to W$, where $\alpha^{\pm} :V^{\pm} \to W^{\pm}$ are linear maps. 

\item $[1] : \G{V}^{\zz/2\zz} \to \G{V}^{\zz/2\zz}$ denotes the self-equivalence of $\G{V}^{\zz/2\zz}$ given by change of grading: 
\[
V[1]^{\pm}=V^{\mp} \mbox{ on objects and $\alpha[1]^{\pm}=\alpha^{\mp}$ on morphisms}.
\]
\item $\G{V}^{\zz/2\zz}_{\T{iso}}$ is the category with the same objects as $\G{V}^{\zz/2\zz}$ and morphisms given by $\alpha = \ma{cc}{\al^+ & 0 \\ 0 & \al^-}$ with $\alpha^{\pm}$ {\it invertible}.
\end{enumerate}
\end{notation}
\begin{dfn}\label{determinant, graded vector spaces}
$|\cdot |$ is the functor from $\G{V}^{\zz/2\zz}_{\T{iso}}$ to $\G{L}$ given by
\[
|V|=|V^+|\otimes |V^-|^{\da} \mbox{ on objects and } |\alpha |=|\alpha^+|\otimes |\alpha^-|^{\da} 
\mbox{ on morphisms}.
\]
\end{dfn}

\subsubsection{Torsion}\label{ss:torsio}
The analogue of the isomorphism \eqref{iso} in the context of $\zz$ (or $\zz/2\zz$) graded vector spaces has the following form.  

Let $V$ be a finite dimensional vector space. The \emph{degree map} $\ep : \La(V) \to \nn_0$ on the exterior algebra over $V$ is defined on homogeneous elements by $\ep : v_1 \wlw v_k \mapsto k$. Suppose that $L$ is a one-dimensional vector space and $t\in L$ a non-zero vector. Then $t^*\in L^*$ will denote the unique vector such that $t^*(t)=1$.

Suppose that $V_i=V_i^+\oplus V_i^-,\ i=1,2$ and $V=V^+\oplus V^-$ are $\zz/2\zz$-graded vector spaces and that we are given grading-preserving linear maps 
\[
f :V_1\rightarrow V_2,\ i:V_2\rightarrow V, \ p : V \to V_1[1]
\]
such that the following six term sequence of finite dimensional vector spaces is exact:
\begin{equation}
\label{eq:six term}
{\mathcal V} : \q \begin{CD}
V_1^+ @>{f^+}>> V^+_2 @>{i^+}>> V^+ \\
@A{p^-}AA & & @VV{p^+}V \\
V^- @<<{i^-}< V_2^- @<<{f^-}< V_1^-
\end{CD}
\end{equation}
For future use, let us introduce the following.
\begin{notation}
We will write the six term exact sequence ${\mathcal V}$ above as the triangle
\[
\xymatrix{
{\mathcal V:}&&V_1\ar[rr]^{f}&&V_2\ar[dl]^i\\
&&&V\ar[ul]^{p[1]}&
}.
\]
and refer to it as an \emph{exact triangle} of $\zz/2\zz$-graded vector spaces.
\end{notation}

\begin{dfn}\label{d:toriso}$\mbox{}$
Suppose that we are given a six term exact sequence $\mathcal V$ of the form  (\ref{eq:six term}).
\noindent Set
\[
\begin{split}
& (V^+_1)_{(0)} := \T{Ker}(f^+) \q (V^+_2)_{(0)} := \T{Ker}(i^+) \q (V^+)_{(0)} := \T{Ker}(p^+) \q \T{and} \\
& (V^-_1)_{(0)} := \T{Ker}(f^-) \q (V^-_2)_{(0)} := \T{Ker}(i^-) \q (V^-)_{(0)} := \T{Ker}(p^-)
\end{split}
\]

\noindent Choose subspaces $(V_i^\pm)_{(1)}\subset V_i^\pm $, $i=1,2$ (resp. $(V^\pm)_{(1)}\subset V^\pm$) complementary to $(V_i^\pm)_{(0)}\subset V_i^\pm $, $i=1,2$ (resp. $(V^\pm)_{(0)}$) and non-zero vectors 
\[
t^\pm_i\in |(V_i^\pm)_{(1)}|,\ i=1,2  \mbox{ (resp.} t^\pm\in |(V^\pm)_{(1)}|).
\]

\noindent The {\bf torsion isomorphism} of $\mathcal V$ is the isomorphism
\[
|\mathcal{V}| : |V_2| \to |V_1|  \ot |V| 
\]
defined by
\[
\begin{split}
& |\mathcal{V}| \big((f^+(t^+_1) \we t^+_2) \ot (f^-(t^-_1) \we t^-_2)^*\big) \\ 
& \q = (-1)^{\mu(\mathcal{V})}(p^-(t^-) \we t^+_1) \ot (p^+ (t^+) \we t^-_1)^* \ot (i^+(t^+_2) \we t^+) \ot (i^-(t^-_2) \we t^-)^*
\end{split}
\]
The sign exponent is defined by
\[
\begin{split}
\mu(\C V) & := \ep(t_2^+) \cd \big( \ep(t_1^-) + \ep(t_1^+) \big) + \ep(t_1^-) \cd \big( \ep(t^+) + \ep(t^-)\big) \\ 
& \q + \ep(t^-) \cd \big(\ep(t_2^+) + \ep(t_2^-)\big) + \ep(t^+) \in \nn_0.
\end{split}
\]

\end{dfn}

\medskip

It is a consequence of \cite[Lemma 2.1.3]{Kaa:JTS} that the torsion isomorphism does not depend on the choices made. For future reference let us note the following simple fact:

\begin{lemma}\label{lemma:det}
Suppose that we are given a six term exact sequence $\mathcal{V}$ of the form
\[
\xymatrix{
{\mathcal V:}&&W\ar[rr]^{f}&&W\ar[dl]\\
&&&0\ar[ul]&
}.
\]
The torsion isomorphism of $\mathcal{V}$ is given by
\[
|\mathcal{V}|= \frac{\det f^-}{\det f^+}.
\]
\end{lemma}
\begin{proof}
This is a straightforward consequence of the definitions.
\end{proof}

\subsection{Fredholm complexes}\label{ss:frecom}
\begin{dfn}
A \emph{Fredholm complex} $\C X$ is a finite cochain complex of (possibly infinite dimensional) vector spaces
\[
\xymatrix{{\mathcal X}:
\ldots\ar[r]&X^k \ar[r]^{d^k}&X^{k+1}\ar[r]^{d^{k+1}} &
X^{k+2}\ar[r]&\ldots
}
\]
where the cohomology groups
\[
H^k({\mathcal X})=\T{Ker}\ d^k/\T{Im}\ d^{k-1}
\]
are finite dimensional. 

The \emph{determinant} of the Fredholm complex ${\mathcal{X}}$ is the graded line
\[
|\mathcal{X}|=|H^+({\mathcal{X}})|\otimes |H^-(\mathcal{X})|^\da
\]
where $H^+(\C X) := \op_{k \in \zz} H^{2k}(\C X)$ and $H^-(\C X) := \op_{k \in \zz} H^{2k+1}(\C X)$.
The \emph{index} of $\mathcal{X}$ is the integer
\[
\T{Ind}(\mathcal{X})=\T{dim}(H^+({\mathcal{X}}))-\T{dim}(H^-({\mathcal{X}}))
\]
\end{dfn}

We let $\C X[1]$ denote the \emph{shift} of the Fredholm complex $\C X$. $\C X[1]$ is again a Fredholm complex with cochains $\C X[1]^{k} := \C X^{k+1}$ and with differentials $d[1]^k := - d^{k+1} : X^{k+1} \to X^{k+2}$, $k \in \zz$.

\begin{dfn}
Let ${\mathcal{X}}$ and ${\mathcal{Y}}$ be finite cochain complexes and let $f : {\mathcal{X}} \to {\mathcal{Y}}$ be a cochain map. The \emph{mapping cone} of $f$ is the cochain complex $C_f$ defined by $C_f^k := X^{k+1} \op Y^k$ and 
\[
d_f^k := \ma{cc}{-d^{k+1}_{\C X} & 0 \\ f^{k+1} & d^k_{\C Y}} : X^{k+1} \op Y^k\to X^{k+2} \op Y^{k+1}
\]
\end{dfn}

Suppose that $f : {\mathcal{X}} \to {\mathcal{Y}}$ is a cochain map of Fredholm complexes. The mapping cone is again a Fredholm complex and it fits into the following \emph{mapping cone triangle:}
\[
\begin{CD}
\De_f : \q \C X @>{f}>> \C Y @>{i}>> C_f @>{p}>> \C X[1]
\end{CD}
\]
where the cochain maps $i : \C Y \to C_f$ and $p : C_f \to \C X[1]$ are given by the inclusions $i^k := \ma{c}{0 \\ 1} : Y^k \to X^{k+1} \op Y^k$ and the projections $p^k := \ma{cc}{1 & 0} : X^{k + 1} \op Y^k \to X^{k+1}$, $k \in \zz$.

By passing to cohomology, the mapping cone triangle associated to $f$ yields an exact triangle of $\zz/2\zz$-graded vector spaces:
\begin{equation}\label{six term}
\begin{CD}
@. H^+(\C X) @>{H^+(f)}>> H^+(\C Y) @>{H^+(i)}>> H^+(C_f) \\
H(\De_f) : \q @. @A{H^-(p)}AA @. @V{H^+(p)}VV \\
@. H^-(C_f) @<<{H^-(i)}< H^-(\C Y) @<<{H^-(f)}< H^-(\C X)
\end{CD}
\end{equation}


\begin{dfn}
Let $f : {\mathcal{X}} \to {\mathcal{Y}}$ be a cochain map of Fredholm complexes. The \emph{torsion isomorphism of $f$} is the torsion isomorphism of $H(\Delta_f)$,
\[
|H(\Delta_f)| : |\mathcal{Y}|\rightarrow |\C X|\otimes |C_f|
\]
(compare with Definition \ref{d:toriso}).
\end{dfn}

\subsection{Joint torsion}
Let $\mathcal{X}$ be a finite cochain complex and let $f:\mathcal{X}\rightarrow \mathcal{X}$ and $g:\mathcal{X}\rightarrow \mathcal{X}$ be two commuting morphisms of $\mathcal{X}$. Remark that $\C X$ is \emph{not} assumed to be a Fredholm complex. Instead, we will assume that the mapping cones $C_f$ and $C_g$ are Fredholm complexes. Since $f$ and $g$ commute (at the level of cochains) we then obtain two cochain maps of Fredholm complexes:
\[
\begin{split}
\delta (g)=\ma{cc}{
g&0\\
0&g
} : C_f\rightarrow C_f
\q \T{and} \q
\delta (f)=\ma{cc}{
f&0\\
0&f
} : C_g \rightarrow C_g
\end{split}
\]
Note that the two mapping cones
$
C_{\delta (f)}
$
and
$
C_{\delta (g)}
$
are in fact isomorphic, the isomorphism is given by the cochain map
$\Phi : C_{\delta (g)} \to C_{\delta (f)}$ defined by
\begin{equation}\label{eq:mapconiso}
\Phi^k := \ma{cccc}{-1 & 0 & 0 & 0 \\ 0 & 0 & 1 & 0 \\ 0 & 1 & 0 & 0 \\ 0 & 0 & 0 & 1} 
: X^{k + 2} \op X^{k+ 1} \op X^{k + 1} \op X^k \to X^{k + 2} \op X^{k+ 1} \op X^{k + 1} \op X^k
\end{equation}
for all $k \in \zz$.

\begin{dfn}\label{d:joytor}
Let $\mathcal{X}$ be a finite cochain complex and let $f:\mathcal{X}\rightarrow \mathcal{X}$ and $g:\mathcal{X}\rightarrow \mathcal{X}$ be two commuting morphisms of $\mathcal{X}$ such that both $C_f$ and $C_g$ are Fredholm complexes. The \emph{joint torsion} of $f$ and $g$ is the non-zero number inducing the automorphism
\[
\T{JT}(\C X; f,g) := |H(\Delta_{\delta (f)})|^{-1} \ci |\Phi| \ci |H(\Delta_{\delta (g)}) | : \ff \to \ff \, ,
\]
where we have used the canonical isomorphisms
\[
|C_f |\otimes |C_f |^\da\cong  |C_g |\otimes |C_g |^\da\cong \ff
\]
to identify the torsion isomorphisms of $\de(g)$ and $\de(f)$ with maps $| H(\De_{\de(g)}) | : \ff \to |C_{\de(g)}|$ and $| H(\De_{\de(f)}) | : \ff \to |C_{\de(f)}|$.
\end{dfn}

\subsubsection{Analyticity of joint torsion}
Let $\{ X^k \}_{k \in \zz}$ be a fixed and finite family of \emph{Hilbert spaces}, thus $X^k = \{0\}$ for all indices outside a finite subset of $\zz$.

Let us for future reference state the following result, which is a consequence of \cite[Theorem 9.1]{KaNe14}. 

Let $\sF$ denote the set of triples $(\mathcal{X};f,g)$, where 
\[
\begin{CD}
\C X : \ldots @>>> X^k @>{d^k}>> X^{k + 1} @>d^{k+1}>> X^{k+2} @>>> \ldots 
\end{CD}
\]
is a cochain complex with each $d^k : X^k \to X^{k + 1}$ a bounded operator and with $f, g : \C X \to \C X$ commuting cochain maps such that $f^k, g^k : X^k \to X^k$ are bounded operators and such that $C_f$ and $C_g$ are Fredholm. We can realize the elements of $\sF$ as a set of bounded operators $(d,f,g)$ on the fixed $\zz$-graded Hilbert space $X = \op_{k \in \zz} X^k$ and hence endow $\sF$ with the induced topology coming from the \emph{operator norm}.

\begin{thm}\label{th:analyticity}
The map
\[
\sF \ni (\mathcal{X};f,g)\to \T{JT}( \C X; f,g) \in \cc^*
\]
is analytic. Thus, for any analytic map $\al : U \to \sF$ defined on an open subset $U \su \cc$ we have that $\T{JT} \ci \al : U \to \cc^*$ is analytic.
\end{thm}

The following variant of Theorem \ref{th:analyticity} also holds:

\begin{thm}\label{t:conjoy}
The map
\[
\sF \ni (\mathcal{X};f,g)\to \T{JT}( \C X; f,g) \in \cc^*
\]
is continuous.
\end{thm}

\begin{remark}
We could have defined the joint torsion in the context of $\zz/2\zz$-graded Fredholm complexes instead of in the more restricted context of finite $\zz$-graded Fredholm complexes. The reason for staying with finite $\zz$-graded Fredholm complexes is that we do not have a proof of Theorem \ref{th:analyticity} for $\zz/2\zz$-graded Fredholm complexes.
\end{remark}

\section{Joint torsion of commuting operators}\label{s:Joint torsion of commuting operators}
\emph{Throughout this section $A = (A_1,\ldots,A_n) \in \sL(\sH)^n$ will denote a commuting tuple of bounded operators on
a Hilbert space $\sH$.} Thus, we have the relation $A_i A_j - A_j A_i = 0$ for all $i,j \in \{1,\ldots,n\}$. Given $\la\in \cc^n$, $A-\la$ will denote the n-tuple $(A_1-\la_1,\ldots,A_n-\la_n) $. We will start by recalling some basic constructions and facts.

\subsection{The Koszul complex}\label{ss:Koszul}
Let $\La(\cc^n)$ denote the exterior algebra on $n$ generators $e_1,\ldots,e_n$. 

For each subset $I = \{i_1,\ldots,i_k\} \su \{1,\ldots,n\}$ with $i_1 < \ldots < i_k$, let $e_I := e_{i_1} \wlw e_{i_k} \in \La(\cc^n)$, where $\we : \La(\cc^n) \ti \La(\cc^n) \to \La(\cc^n)$ denotes the wedge product.

The exterior algebra is then a $\zz$-graded algebra with respect to the decomposition
\[
\La(\cc^n) = \op_{k \in \zz} \La^k(\cc^n),
\]
where $\La^k(\cc^n) := \{0\}$ for $k \notin \{0,\ldots,n\}$ and $\La^k(\cc^n) := \T{span}_\cc\big\{ e_I \, | \, I \su \{1,\ldots,n\} \, , \, |I| = k\big\}$ for $k \in \{0,\ldots,n\}$.

For each $j \in \{1,\ldots,n\}$, the interior multiplication with the $j^{\T{th}}$ generator is denoted by 
\[
\begin{split}
& \ep_j^* : \La(\cc^n) \to \La(\cc^n) \\ 
& \ep_j^* : e_I \mapsto \fork{ccc}{
0 & \T{for} & j \notin I \\
(-1)^{m-1} e_{i_1} \wlw \wih{e_{i_m}} \we e_{i_{m+1}} \wlw e_{i_k} & \T{for} & j = i_m. }
\end{split}
\]
This linear map has degree $-1$ with respect to the above $\zz$-grading.

\begin{dfn}
By the \emph{Koszul complex of $A$} we will understand the finite cochain complex of Hilbert spaces given by the following data:
\begin{enumerate}
\item The Hilbert space $K^k(A,\sH) := \sH \ot_\cc \La^{-k}(\cc^n)$ for each $k \in \zz$.
\item The differential of degree one,
\[
d_A := \sum_{j=1}^n A_j \ot \ep_j^* : K^k(A,\sH) \to K^{k+1}(A,\sH).
\]
\end{enumerate}
We will use the notation $K(A,\sH)$ for the Koszul complex and the notation $H^k(A,\sH)$, $k \in \zz$ for the cohomology groups of
$K(A,\sH)$.

We will say that a commuting n-tuple $A$ is {\em Fredholm} if $K(A,\sH)$ is  Fredholm. In this case, the \emph{index} of $A$ is the index of $K(A,\sH)$,
\[
\T{Ind}(A) := \sum_{k \in \zz} (-1)^k \T{dim}_{\cc}\big( H^k(A,\sH) \big)
\]
\end{dfn}

\begin{remark}$\mbox{}$
\begin{enumerate}
\item It is straightforward to see that $d_A^{k+1} d_A^k = 0$ for all $k \in \zz$.
\item When $A$ is Fredholm we have that $d_A^k : K^k(A,\sH) \to K^{k + 1}(A,\sH)$ has closed image for all $k \in \zz$. This is a consequence of \cite[Corollary 6.2]{Cur:FOD}.
\item This is not quite the conventional definition of the Koszul complex. The reason for using interior instead of exterior multiplication on the exterior algebra stems from the fact that we want the mapping cone of a bounded operator $B : \sH \to \sH$ to be isomorphic to the Koszul complex $K(B,\sH)$ (without the dimension shift). Indeed, with the present convention, we have the following (see \cite[Lemma 2.3]{KaNe12}):
\end{enumerate} 
\end{remark}

\begin{lemma}\label{l:mapkos}
Let $B : \sH \to \sH$ be a bounded operator such that $B A_j = A_j B$ for all $j \in \{1,\ldots,n\}$. Then the mapping cone of the cochain map
\[
K^*(B) := B \ot 1 : K^*(A,\sH) \to K^*(A,\sH)
\]
is cochain isomorphic to the Koszul complex $K^*\big( (B,A),\sH \big)$.
\end{lemma}

Recall that the \emph{Taylor spectrum} of $A$ is the set
\[
\T{Sp}(A) := \big\{\la \in \cc^n \, | \, H^*(A - \la,\sH) \neq \{0\} \big\}
\]
The Taylor spectrum is a compact non-empty subset of $\cc^n$, \cite[Theorem 3.1]{Tay:JSC}. The essential Taylor spectrum of $A$ is the set
\[
\spe := \big\{\la \in \cc^n \, | \, {A-\la} \T{ is not Fredholm} \big\}
\]

The unital commutative ring of germs of holomorphic functions on neighborhoods of $\T{Sp}(A)$ will be denoted by $\hoc$. The elements in $\hoc$ are thus equivalence classes of holomorphic functions $f : U \to \cc$, where $U$ is an open subset of $\cc^n$ containing $\T{Sp}(A)$ . 

The commuting tuple of bounded operators $A = (A_1,\ldots,A_n)$ on the Hilbert space $\sH$ provides us with a holomorphic functional calculus. To be more precise, the following holds (see \cite[Theorem 4.8]{Tay:ACS}).

\begin{thm}\label{t:anafun}
Let $\C A \su \sL(\sH)$ denote the smallest unital $\cc$-algebra which contains the bounded operators $A_1,\ldots,A_n$. There exists a unital homomorphism $\hoc \to \C A''$ (the double commutant of $\C A$), $f \mapsto f(A)$ such that $z_i \mapsto A_i$. Furthermore, whenever $f = (f_1,\ldots,f_m) : \T{Sp}(A) \to \cc^m$ is holomorphic (i. e. $f_k\in \hoc ,\ k=1,\ldots,m$) the following identity holds:
\[
\T{Sp}(f(A)) = f\big(\T{Sp}(A) \big).
\]
\end{thm}

\medskip

\noindent {In particular, the above holomorphic functional calculus allows us to consider the Hilbert space $\sH$ as a left module over $\hoc$. }

\begin{notation}\label{n:hol}$\mbox{}$
\begin{enumerate}
\item Given $f\in \hoc$, we will denote the operator $f(A)$ on $\sH$ by $f$. Given $g\in \hoc^m$, $K^*(g,\sH)$ will denote the Koszul complex of $g(A)=(g_1(A),\ldots,g_m(A))$ and $H^*(g,\sH)$ will denote the cohomology of $K^*(g,\sH)$.
\item
For each $\la \in \T{Sp}(A)$,  $\sH_\la$ denotes the localization of the module $\sH$ with respect to the prime ideal $\G p_\la :=\{f \in \hoc\mid f(\la) = 0\} \su \hoc$.

\item Let $m \in \nn$ and suppose that we are given $g = (g_1,\ldots,g_m)\in \hoc^m$. We set
\[
Z(g) := \big\{ \la \in \T{Sp}(A) \mid g_1(\la) = \ldots = g_m(\la) = 0 \big\}.
\]
\end{enumerate}
\end{notation}

\begin{example}\label{spectrum:polydisc}(\cite{Cur:FOD}) Let $H^2(\B D^n)$ be the Hardy-space over the polydisc $\mathbb{D}^n=\{z\in\cc^n\mid |z_j| < 1, \T{ for all } j = 1,\ldots,n \}$, and $A=(T_{z_1},\ldots,T_{z_n})$ be the $n$-tuple of multiplication operators by the coordinate functions $z_1,\ldots,z_n$ on $\cc^n$. Then,
\begin{enumerate}
\item $\T{Sp}(A)=\ov{\mathbb{D}^n}$;
\item for a function $f$ holomorphic in a neighbourhood of $\ov{\mathbb{D}^n}$, $f(A)$ coincides with the Toeplitz operator $T_f$ of multiplication by $f$;
\item $\T{Sp}_{\T{ess}}(A)=\partial \mathbb{D}^n$;
\item for $\la \in \B D^n$, $\T{Ind}(A-\la)= 1$;
\item for $\la \in \B D^n$, $H^k\big(A - \la,H^2(\B D^n) \big) = \{0\}$ for all $k \in \zz\sem \{0\}$.
\end{enumerate}
\end{example}
The next result is a consequence of \cite[Proposition 4.5]{KaNe14}.

\begin{prop}\label{p:deckos}
Suppose that $g\in \hoc^m$ is such that $g(A) = (g_1(A),\ldots,g_m(A))$ is Fredholm. Then $Z(g)$ is finite and the homomorphism of modules $\sH \to \op_{\la \in Z(g)} \sH_\la$, $\xi \mapsto \{\xi/1\}_{\la \in Z(g)}$, induces an isomorphism of cohomology groups,
\[
H^*\big( g,\sH \big) \cong H^*\big( g,\op_{\la \in Z(g)} \sH_\la \big) \cong \op_{\la \in Z(g)} H^*\big( g,\sH_\la \big).
\]
\end{prop}

\subsection{Joint torsion transition numbers}
Let $m \in \nn$ and let $\inn{m} := \{1,\ldots,m\}$. Fix an element $g=(g_1,\ldots,g_{m})\in \hoc^{m}$. For a subset $J = \{j_1,\ldots,j_k\} \su \inn{m}$ with $1 \leq j_1 < \ldots < j_k \leq m$, we define
\[
g_J := (g_{j_1},\ldots, g_{j_k}) \in \hoc^k
\]
The following assumption will remain in effect throughout this subsection:
\medskip

\noindent{\emph{Let $i,j \in \inn{m}$ and suppose that the commuting tuples $g_{\inn{m}\sem \{i\}}(A)$ and $g_{\inn{m} \sem \{j\}}(A)$ are Fredholm.}}
\medskip

It then follows by Lemma \ref{l:mapkos} that the mapping cones of the commuting cochain maps
\begin{equation}\label{eq:cocfre}
K^*(g_i) \T{ and } K^*(g_j) : K^*( g_{\inn{m}\sem \{i,j\}},\sH) \to K^*( g_{\inn{m}\sem \{i,j\}},\sH)
\end{equation}
are Fredholm complexes. In particular, the following definition makes sense:

\begin{dfn}\label{d:joytra}
The \emph{joint torsion transition number} of $g(A)$ (in position $i,j$) is defined as the joint torsion of the cochain maps in \eqref{eq:cocfre}. It is denoted by
\[
\tau_{i,j}(g(A)) := \T{JT}\big( K^*\big( g_{\inn{m}\sem \{i,j\}},\sH\big) ; g_i , g_j \big) \in \cc^*
\]
\end{dfn}


%

\subsubsection{The torsion line bundle}\label{ss:Torsion line bundle}
As in the last subsection, let $g = (g_1,\ldots,g_m) \in \hoc^m$ be fixed. Recall that $\inn{m} := \{1,\ldots,m\}$.

For each $i \in \{1,\ldots,m\}$, define the open subset
\[
U_i=\big\{\mu\in \cc^m\mid (g - \mu)_{\inn{m}\sem\{i\}}(A) \T{ is Fredholm}\big\}
\]

For each $i,j \in \{1,\ldots,m\}$ and each $\mu \in U_i\cap U_j$, we then have the joint torsion transition number
\[
\tau_{i,j}(g(A))(\mu) := \tau_{i,j}\big( g(A) - \mu \big) \in \cc^*
\]
See Definition \ref{d:joytra}.

As a consequence of \cite[Lemma 3.3.3]{Kaa:JTS} these functions satisfy the transition identities of a line bundle, thus
\[
\tau_{i,j}\big(g(A)\big) \cd \tau_{j,k}\big( g(A)\big) = \tau_{i,k}\big(g(A)\big) \q \T{and} \q 
\tau_{i,j}\big(g(A)\big) = \tau_{j,i}\big(g(A)\big)^{-1},
\]
for all $i,j,k \in \{1,\ldots,m\}$.

Furthermore, by Theorem \ref{th:analyticity}, each $\tau_{i,j}(g(A)) : U_i\cap U_j \to \cc^*$ is analytic, hence the following makes sense:

\begin{dfn}
The \emph{torsion line bundle} of $g(A)$ is the analytic line bundle on $U :=\cup_{i=1}^m U_i \su \cc^m$ with transition functions $\tau_{i,j}(g(A)) : U_i \cap U_j \to \cc^*$, $i,j=1,\ldots, m$.
\end{dfn}

\subsection{Local indices}\label{ss:local indices}
Let $\Om := \T{Sp}(A)^\ci$ denote the interior of the Taylor spectrum and let $\C O$ denote the sheaf of analytic functions on $\Om$. Fix an element $g = (g_1,\ldots,g_n) \in \hoc^n$. Remark that the number of holomorphic functions in $g$ coincides with the number of operators in the tuple $A = (A_1,\ldots,A_n)$.
\medskip

\noindent {\it Throughout this subsection we suppose that the commuting tuple $g(A) = (g_1(A),\ldots,g_n(A))$ is Fredholm.}
\medskip

\noindent This means precisely that the intersection $Z(g) \cap \spe$ is the trivial set. In particular, there is a well-defined index $\T{Ind}(A - \la) \in \zz$ for each $\la \in Z(g)$.

Now, applying Proposition \ref{p:deckos}, it also follows from the Fredholmness of $g(A)$ that the set of common zeroes $Z(g) \su \T{Sp}(A)$ is finite. In particular, each element $\la \in \Om \cap Z(g)$ is an isolated zero for the holomorphic map $g|_{\Om} : \Om \to \cc^n$. This is equivalent to the finite dimensionality of the quotient vector space $\C O_\la/\inn{g}_\la$, where $\inn{g}_\la \su \C O_\la$ denotes the ideal generated by the analytic functions $g_1,\ldots,g_n : \Om \to \cc$ in the stalk $\C O_\la$ of the sheaf $\C O$ at $\la$. See \cite[Chapter 5,\S 1.2]{GrRe:CAS}.

\begin{dfn}
The \emph{local degree}, or (multiplicity) of $g$ at $\la \in \Om$ is the dimension $\T{dim}_{\cc}\big(\C O_\la/\inn{g}_\la \big) \in \nn \cup \{0\}$. The local degree is denoted by $m_\la(g)$.
\end{dfn}

\begin{dfn}
The \emph{local index} of $g(A)$ at $\la \in \spa$ is the Euler characteristic of the Koszul complex $K(g,\sH_\la)$. The local index is denoted by $\T{Ind}_\la(g(A)) \in \zz$, thus
\[
\T{Ind}_\la(g(A)) := \sum_{i\in\zz} (-1)^{i} \T{dim}_{\cc}\big( H^i(g,\sH_\la) \big).
\]
\end{dfn}

Remark that the finite dimensionality of the cohomology groups $H^i(g,\sH_\la)$, $i \in \zz$ is non-obvious. This is a consequence of Proposition \ref{p:deckos} which also implies the identity
\[
\T{dim}_{\cc}H^i(g(A),\sH) = \sum_{\la \in Z(g)} \T{dim}_{\cc}\big( H^i(g,\sH_\la)\big).
\]
In particular we have that
\[
\T{Ind}(g(A)) = \sum_{\la \in Z(g)} \T{Ind}_\la(g(A)).
\]

The following "local" index theorem therefore yields the "global" index theorem of \cite[Theorem 10.3.13]{EsPu:SDA}. 

\begin{thm}\label{t:locind}(see \cite[Theorem 8.5]{KaNe12})
Suppose that $g(A)$ is Fredholm and that $\la \in Z(g)$. The local index at $\la$ is then given by
\[
\T{Ind}_\la(g(A)) = m_\la(g) \cd \T{Ind}(A - \la).
\]
\end{thm}

\begin{remark}
By homotopy invariance of the Fredholm index, $\T{Ind}(A - \la) = 0$ when $\la \in \pa (\spa )\cap (\cc^n\sem \spe)$. The right hand side of the equation in the above theorem should therefore be understood as $0$ in this case, even though the local degree, $m_\la(g)$, is \emph{not} defined.
\end{remark}

\section{Multiplicative Lefschetz numbers}\label{s:Multiplicative Lefschetz numbers}
Let $\C X$ be a Fredholm complex over a fixed field $\B F$ of characteristic $0$, and let $f : \C X \to \C X$ be a cochain map. 
\medskip

\emph{Suppose that the induced maps $H^{+}(f) : H^{+}(\C X) \to H^{+}(\C X)$ and $H^{-}(f) : H^{-}(\C X) \to H^{-}(\C X)$ are invertible.}
\medskip

\begin{dfn}
The \emph{multiplicative Lefschetz number} of $f : \C X \to \C X$ is the invertible element
\[
M(\C X; f) := \frac{\T{det}( H^{+}(f)) }{\T{det}( H^{-}(f))} \in \ff^*
\]
\end{dfn}

\subsection{Lefschetz numbers of Koszul complexes}\label{ss:Lefschetz numbers of Koszul complexes}
Let $A = (A_1,\ldots,A_n)$ be a commuting tuple of bounded operators on a Hilbert space $\sH$, and let $h_1,\ldots,h_{n-1},f,g\in \hoc $. We will use the notation
\[
(h,f) :=(h_1,\ldots,h_{n-1},f)\in \hoc^n
\]
(and similarly for $(h,g)$). Throughout this subsection we will suppose that:
\medskip

\noindent \emph{The sets $Z(h,f) \cap \T{Sp}_{\T{ess}}(A)$ and $Z(h,g) \cap \T{Sp}_{\T{ess}}(A)$ and $Z(h,f,g)$ are empty.}
\medskip

It follows from this assumption that the Koszul complexes $K\big( (h,f), \sH \big)$ and $K\big( (h,g), \sH \big)$ are Fredholm. Furthermore, we have that the cohomology of the Koszul complex $K\big( (h,f,g),\sH \big)$ is trivial.

Notice now that $f$ and $g$ induce cochain maps
\begin{equation}\label{eq:koscha}
\begin{split}
& f := K^*(f) : K\big( (h,g),\sH \big) \to K\big( (h,g),\sH \big) \q \T{and} \\
& g := K^*(g) : K\big( (h,f),\sH \big) \to K\big( (h,f),\sH \big)
\end{split}
\end{equation}
of Koszul complexes by means of the holomorphic functional calculus. As a consequence of the exactness of $K\big( (h,f,g),\sH \big)$ we obtain that the induced maps $H^*(f) : H^*\big( (h,g),\sH \big) \to H^*\big( (h,g),\sH\big)$ and $H^*(g) : H^*\big( (h,f), \sH\big) \to H^*\big( (h,f), \sH\big)$ are invertible.

The involved quantities of the next proposition are therefore well-defined.

\begin{thm}\label{t:nonsin}
The following identities hold,
\[
\T{JT}\big( K(h,\sH) ; f,g \big) = 
\frac{M\big(K\big( (h,g),\sH \big); f \big)}{M\big(K\big( (h,f),\sH \big) ; g \big)} =
\frac{\prod_{\la \in Z(h,g)} f(\la)^{m_\la(h,g) \cd \T{Ind}(A - \la)}}
{\prod_{\mu \in Z(h,f)} g(\mu)^{m_\mu(h,f) \cd \T{Ind}(A - \mu)}}.
\]
\end{thm}
\begin{proof}
The first identity is an immediate consequence of the definition of the joint torsion and of Lemma \ref{lemma:det} (see also \cite[Theorem 3.4.1]{Kaa:JTS}).

To prove the second identity, it suffices to show that
\[
M\big(K\big( (h,g),\sH \big); f \big) = \prod_{\la \in Z(h) \cap Z(g)} f(\la)^{m_\la(h,g) \cd \T{Ind}(A - \la)}.
\]

Let $i \in \{-n,\ldots,0\}$. Since $Z(h) \cap Z(g) \cap \spe = \emp$, the Koszul cohomology group $H^i\big( (h, g),\sH \big)$ is finite dimensional over $\cc$. Let $H^i(A) := \big(H^i(A_1),\ldots,H^i(A_n)\big)$ denote the commuting tuple of linear operators on $H^i\big((h,g),\sH\big)$ induced by $A = (A_1,\ldots,A_n)$. For each $\la \in \cc^n$, let $H^i\big((h,g),\sH\big)(\la) \su H^i\big((h,g),\sH\big)$ denote the generalized eigenspace of the commuting tuple $H^i(A)$. Thus,
\[
\begin{split}
& H^i\big((h,g),\sH\big)(\la) \\
& \q := \big\{\xi \in H^i\big((h,g),\sH\big) \, | \, \exists m \in \nn : H^i(A_j - \la_j)^m(\xi) = 0 \, \, \forall j \in \{1,\ldots,n\} \big\}.
\end{split}
\]

Recall now that the finite dimensionality of $H^i((h,g),\sH)$ implies that
\[
\op_{\la \in Z(h,g) } H^i\big((h,g),\sH\big)(\la) \cong H^i\big((h,g),\sH\big),
\]
and furthermore that each component $H^i\big((h,g),\sH\big)(\la) \su H^i\big((h,g),\sH\big)$ admits a basis in which each of the restrictions
\[
H^i(A_j)(\la) : H^i\big((h,g),\sH\big)(\la) \to H^i\big((h,g),\sH\big)(\la)
\]
is upper triangular with only $\la_j$ on the diagonal. 
%

For each $\la \in Z(h) \cap Z(g)$, let $H^i(f(A))(\la) : H^i((h,g),\sH)(\la) \to H^i((h,g),\sH)(\la)$ denote the restriction of the isomorphism $H^i(f(A)) : H^i((h,g),\sH) \to H^i((h,g),\sH)$. It then follows immediately from the above that
\begin{equation}\label{eq:detzer}
\T{det}\big( H^i(f(A)) \big) = \prod_{\la \in Z(h) \cap Z(g)} \T{det}\big( H^i(f(A))(\la) \big).
\end{equation}

The next lemma gives a computation of the determinant $\T{det}\big( H^i(f(A))(\la)\big)$ for each $\la \in Z(h) \cap Z(g)$.

\begin{lem}\label{l:detpoi}
Let $\la \in Z(h) \cap Z(g)$. The isomorphism $H^i(f(A))(\la) : H^i\big((h,g),\sH\big)(\la) \to H^i\big((h,g),\sH\big)(\la)$ can be represented by an upper triangular matrix having $f(\la) \in \cc^*$ as its only diagonal entry. In particular,
\begin{equation}
\T{det}( H^i(f(A))(\la) ) = f(\la)^{\T{dim}_{\cc} H^i((h,g),\sH)(\la) }.
\end{equation}
\end{lem}
\begin{proof}
Suppose first that $f \in \hoc$ is the restriction of a polynomial $p$ in the variables $z_1,\ldots,z_n : \cc^n \to \cc$. In this case $H^i(f(A))(\la) :  H^i((h,g),\sH)(\la) \to H^i((h,g),\sH)(\la)$ is given by the polynomial $p(H^i(A)(\la))$ where each variable $z_j$ has been replaced by the linear operator $H^i(A_j)(\la) : H^i((h,g),\sH)(\la) \to H^i((h,g),\sH)(\la)$. Choose a basis for $H^i((h,g),\sH)(\la)$ in which each of the operators $H^i(A_j)(\la)$ is represented by an upper triangular matrix having $\la_j$ as the only diagonal entry. Represented in this basis $p(H^i(A)(\la))$ is an upper triangular matrix with $p(\la)$ as its only diagonal entry. This proves the claim of the lemma in this case.

To treat the general case, note that the action of $\hoc$ on $H^i((h,g),\sH)(\la)$ given by $k \mapsto H^i(k(A))(\la)$ factorizes through the local $\cc$-algebra $\C O_\la$ of convergent power series near $\la$. This is the content of \cite[Proposition 4.4]{KaNe12}. Since the endomorphisms $H^i(A_j - \la_j)(\la) : H^i((h,g),\sH)(\la) \to H^i((h,g),\sH)(\la)$ are nilpotent, $j \in \{1,\ldots,n\}$, this yields the existence of a polynomial $p \in \cc[z_1,\ldots,z_n]$ with $H^i(f(A))(\la) = H^i(p(A))(\la)$ and with $p(\la) = f(\la)$. To see this, remark that the nilpotency implies that the action of $\C O_\la$ on $H^i(A_j - \la_j)(\la)$ factorizes through the quotient ring $\C O_\la/ (\G{m}_\la)^m \C O_\la$ for some $m \in \nn$, where $\G{m}_\la \su \C O_\la$ denotes the unique maximal ideal in the local $\cc$-algebra $\C O_\la$. But each element in $\C O_\la/(\G{m}_\la)^m \C O_\la$ can be represented by a polynomial.

Since it now has been established that $H^i(f(A))(\la) = H^i(p(A))(\la)$ for some polynomial with $p(\la) = f(\la)$, the first part of the proof yields the general result of the lemma.
\end{proof}

Let $\la \in Z(h) \cap Z(g)$ be fixed. To continue the proof of Theorem \ref{t:nonsin}, remark that it follows by \cite[Proposition 4.5]{KaNe12}, that the vector spaces $H^i\big((h,g),\sH \big)(\la)$ and $H^i\big((h,g),\sH_\la \big)$ are isomorphic. Recall in this respect that $\sH_\la$ denotes the localization of the module $\sH$ over $\hoc$ at the prime ideal $\G{p}_\la := \{k \in \hoc \,| \, k(\la) = 0\}$. In particular, it follows from Theorem \ref{t:locind} that
\begin{equation}\label{eq:indide}
\begin{split}
& \T{dim}_{\cc} H^{+}\big( (h,g),\sH\big)(\la) - \T{dim}_{\cc} H^{-}\big( (h,g),\sH\big)(\la) \\
& \q = \T{Ind}_{\la}\big( (h,g)(A) \big) = m_\la(h,g) \cd \T{Ind}(A - \la).
\end{split}
\end{equation}

The desired multiplicative Lefschetz number $M\big(K((h,g),\sH); f\big) \in \cc^*$ can now be computed as follows,
\[
\begin{split}
& M\big(K((h,g),\sH); f\big)
= \T{det}\big(H^{+}(f)\big) \cd 
\T{det}\big(H^{-}(f)\big)^{-1} \\
& \q =  \prod_{\la \in Z(h) \cap Z(g)} \Big( \T{det}\big(H^{+}(f)(\la)\big)
\cd \T{det}\big(H^{-}(f)(\la)\big)^{-1} \Big) \\
& \q = \prod_{\la \in Z(h) \cap Z(g)}\Big( f(\la)^{\T{dim}_{\cc} H^{+}((h,g),\sH)(\la)}  \cd f(\la)^{-\T{dim}_{\cc} H^{-}((h,g),\sH)(\la)}\Big) \\
& \q = \prod_{\la \in Z(h) \cap Z(g)} f(\la)^{\T{Ind}_\la ( (h, g)(A) )}
= \prod_{\la \in Z(h) \cap Z(g)} f(\la)^{m_\la(h,g) \cd \T{Ind}(A-\la)},
\end{split}
\]
where the second identity follows from \eqref{eq:detzer}, the third identity follows by Lemma~\ref{l:detpoi}, and the final two identities follow from \eqref{eq:indide}. This proves the claim of Theorem~\ref{t:nonsin}.
\end{proof}

\section{Localization of the joint torsion}\label{s:Localization of the joint torsion}
Let $n \in \nn$. For each open set $U \su \cc^n$ we let $\C O(U)$ denote the unital $\cc$-algebra of holomorphic functions on $U$ with values in $\cc$.

For each compact set $K \su U$, we let $\sC(K)$ denote the unital $\cc$-algebra of continuous functions $f : K \to \cc$ such that the restriction to the interior $f|_{K^\ci} : K^\ci \to \cc$ is holomorphic. The unital $\cc$-algebra $\sC(K)$ becomes a Banach algebra when equipped with the supremum norm $\| \cd \|_\infty : f \mapsto \sup_{z \in K} |f(z)|$. We let
\[
r_K : \C O(U) \to \sC(K)
\]
denote the restriction homomorphism.

\begin{dfn}
Let $m \in \nn$ and let $V \su \cc^m$. We say that a map $\al : V \to \C O(U)$ is \emph{holomorphic} when the composition
\[
r_K \ci \al : V \to \sC(K)
\]
is holomorphic for each compact set $K \su U$.
\end{dfn}

Let us fix a commuting $n$-tuple $A = (A_1,\ldots,A_n)$ of bounded operators on the Hilbert space $\sH$.

\begin{prop}\label{p:holfuncal}
Let $U \ssu \T{Sp}(A)$ be an open set and let $\al : V \to \C O(U)$ be holomorphic. Then the map $V \to \sL(\sH)$, $w \mapsto \al(w)(A)$ is holomorphic in operator norm.
\end{prop}
\begin{proof}
Choose a compact set $K \su U$ such that $\T{Sp}(A) \su K^\ci$. By \cite[Theorem 4.3]{Tay:ACS} the map $\sC(K) \to \sL(\sH)$, $f \mapsto f(A)$ is a bounded operator. Since the composition $r_K \ci \al : V \to \sC(K)$ is holomorphic by definition, this proves the proposition.
\end{proof}

Let us apply the notation
\[
\Om := \T{Sp}(A)^\ci
\]
for the interior of the spectrum of $A$.

\begin{thm}\label{t:locjoytor}
Let $U \ssu \T{Sp}(A)$ be an open set and let $h : U \to \cc^{n - 1}$ and $f,g : U \to \cc$ be holomorphic maps.

Suppose that $Z(h,f) \cap \T{Sp}_{\T{ess}}(A)$ and $Z(h,g) \cap \T{Sp}_{\T{ess}}(A)$ are empty.

Then the joint torsion of $\big( K(h,\sH); f,g\big)$ is given by
\begin{equation}\label{eq:locjoytor}
\T{JT}\big(  K(h,\sH); f,g\big) =
\lim_{w \to 0}\left( \frac{\prod_{\la \in Z(h_w,g_w) \cap \Om} f_w(\la)^{m_{\la}(h_w,g_w) \cd \T{Ind}(A - \la)}}
{\prod_{\mu \in Z(h_w,f_w) \cap \Om} g_w(\mu)^{m_{\mu}(h_w,f_w) \cd \T{Ind}(A - \mu)}} \right)
\end{equation}
for any holomorphic map $V \to \C O(U)^{n + 1}$, $w \mapsto (h_w,f_w,g_w)$ such that
\begin{enumerate}
\item $0 \in V$ and $(h_0,f_0,g_0) = (h,f,g)$.
\item $Z(h_w,f_w,g_w) \cap \T{Sp}(A) = \emptyset$ for all $w \in V\sem \{0\}$.
\item $Z(h_w,f_w) \cap \T{Sp}_{\T{ess}}(A)$ and $Z(h_w,g_w) \cap \T{Sp}_{\T{ess}}(A)$ are empty for all $w \in V$.
\end{enumerate}
\end{thm}
%

\begin{proof}
We first prove the existence of a holomorphic map $V \to \C O(U)^{n + 1}$ with the properties $(1)$, $(2)$, and $(3)$.

Define the strictly positive numbers
\[
\begin{split}
& \de_0 := \inf\big\{ |f(z)| \mid z \in Z(h) \cap \T{Sp}_{\T{ess}}(A) \big\} \q \T{and} \\
& \de_1 := \inf\big\{ |f(z)| \mid z \in Z(h,f) \cap \T{Sp}(A) \T{ and } f(z) \neq 0 \big\}
\end{split}
\]
Consider the open ball
\[
\B B_\de(0) := \big\{ w \in \cc \mid |w| < \de \big\} 
\]
of radius $\de := \inf\{\de_0,\de_1\}$ and center $0 \in \cc$.

It can then be verified that the holomorphic map $\B B_\de(0) \to \C O(U)^{n+1}$, $w \mapsto (h,f-w,g)$ has the properties $(1)$, $(2)$, and $(3)$.

Let now $V \to \C O(U)^{n + 1}$ be any holomorphic map which satisfies $(1)$, $(2)$, and $(3)$. By Theorem \ref{th:analyticity} and Proposition \ref{p:holfuncal} we have that
\[
\T{JT}\big( K(h,\sH); f,g\big) = \lim_{w \to 0} \T{JT}\big( K(h_w,\sH); f_w, g_w \big)
\]
However, by Theorem \ref{t:nonsin} we may compute the joint torsion on the right hand side,
\[
\T{JT}\big( K(h_w,\sH); f_w, g_w \big) = \frac{\prod_{\la \in Z(h_w,g_w) \cap \Om} f_w(\la)^{m_{\la}(h_w,g_w) \cd \T{Ind}(A - \la)}}
{\prod_{\mu \in Z(h_w,f_w) \cap \Om} g_w(\mu)^{m_{\mu}(h_w,f_w) \cd \T{Ind}(A - \mu)}}
\]
for all $w \in V\sem \{0\}$. This proves the theorem.
\end{proof}

\begin{remark}\label{r:expjoytor}
It follows from the proof of Theorem \ref{t:locjoytor} that the joint torsion $\T{JT}(K(h,\sH);f,g)$ can be computed more explicitly as the limit
\[
\T{JT}(K(h,\sH);f,g) = \lim_{w \to 0} \left( \frac{\prod_{\la \in Z(h,g) \cap \Om} (f(\la) - w)^{m_\la(h,g) \cd \T{Ind}(A - \la)}}{\prod_{\mu \in \Om \cap Z(h) \cap f^{-1}(\{w\})} g(\mu)^{m_\mu(h,f - w) \cd \T{Ind}(A - \mu)}  } \right)
\]
where $w \in \cc$ approaches zero in the Euclidean metric.
\end{remark}

For each $\la \in \cc^n$ and each $\ep > 0$, we apply the notation
\[
\B D^n_\ep(\la) := \big\{ z \in \cc^n \mid |z_j - \la_j| < \ep \T{ for all } j \in \{1,\ldots,n\} \big\} 
\]
for the open polydisc with radius $\ep >0$ and center $\la$.

\begin{thm}\label{t:joystalk}
Let $U \su \cc^n$ be open and let $h : U \to \cc^{n-1}$ and $f,g : U \to \cc$ be holomorphic.

Suppose that $\la \in U$ is an isolated point in $Z(h,f) \cup Z(h,g)$.

Then the limit
\[
c_\la(h;f,g) = \lim_{w \to 0} \left( 
\frac{\prod_{\mu \in Z(h_w,g_w) \cap \B D^n_{\ep/2}(\la)} f_w(\mu)^{m_\mu(h_w,g_w)}}
{\prod_{\nu \in Z(h_w,f_w) \cap \B D^n_{\ep/2}(\la)} g_w(\nu)^{m_\nu(h_w,f_w)}}
\right)
\]
exists, for any $\ep > 0$ and any holomorphic map $V \to \C O(\B D^n_\ep(\la))^{n + 1}$, $w \mapsto (h_w,f_w,g_w)$ such that
\begin{enumerate}
\item $\B D^n_\ep(\la) \su U$ and $\B D^n_\ep(\la) \cap \big( Z(h,f) \cup Z(h,g) \big) = \{\la\}$.
\item $0 \in V$ and $(h_0,f_0,g_0) = (h,f,g)$.
\item $Z(h_w,f_w,g_w) \cap \ov{\B D^n_{\ep/2}(\la)} = \emptyset$ for $w\in V\sem\{0\}$.
\item $Z(h_w,f_w) \cap \pa \B D^n_{\ep/2}(\la)$ and $Z(h_w,g_w) \cap \pa \B D^n_{\ep/2}(\la)$ are empty for all $w \in V$.
\end{enumerate}
Furthermore, $c_\la(h;f,g) \in \cc^*$ only depends on the image of $(h,f,g) \in \C O(U)^{n+1}$ in the stalk $\C O_\la^{n+1}$ at $\la \in U$.
\end{thm}
\begin{proof}
Consider the commuting tuple $A := \big( \ep/2 T_{z_1},\ldots, \ep/2 T_{z_n} \big) + \la$ of Toeplitz operators acting on the Hardy space over the polydisc, $H^2(\B D^n)$. It then follows that $\T{Sp}(A) = \ov{\B D^n_{\ep/2}(\la)}$ and that $\T{Sp}_{\T{ess}}(A) = \pa \B D^n_{\ep/2}(\la)$. Furthermore, we have that $\T{Ind}(A - \mu) = 1$ for all $\mu \in \B D^n_{\ep/2}(\la)$, see Example \ref{spectrum:polydisc}.

An application of Theorem \ref{t:locjoytor} then yields that the limit $c_\la(h;f,g)$ exists and coincides with the joint torsion $\T{JT}\big( K(h(A),H^2(\B D^n));f(A),g(A) \big)$.

To see that $c_\la(h;f,g)$ only depends on the value of $(h,f,g) \in \C O(U)$ in the stalk $\C O_\la$ it suffices to check that it is independent of the holomorphic map $V \to \C O(\B D^n_\ep(\la))^{n+1}$ and of $\ep > 0$.

It follows immediately by Theorem \ref{t:locjoytor} that $c_\la(h;f,g)$ is independent of the choice of holomorphic map $V \to \C O(\B D^n_\ep(\la))^{n+1}$. 

Let us thus choose an alternative $\ep_0 > 0$ with $\ep_0 < \ep$. We may then find a holomorphic map $V_0 \to \C O(\B D^n_\ep(\la))$, $w \mapsto (h_w,f_w,g_w)$ such that
\[
\big( Z(h_w,f_w) \cup Z(h_w,g_w) \big) \cap \B D^n_{\ep_0/2}(\la) 
= \big( Z(h_w,f_w) \cup Z(h_w,g_w) \big) \cap \ov{ \B D^n_{\ep/2}(\la) }
\]
and such that $(2)$ is satisfied as well. It is then clear that
\[
\begin{split}
c_\la^\ep(h;f,g) & = \lim_{w \to 0} \left( 
\frac{\prod_{\mu \in Z(h_w,g_w) \cap \B D^n_{\ep/2}(\la)} f_w(\mu)^{m_\mu(h_w,g_w)}}
{\prod_{\nu \in Z(h_w,f_w) \cap \B D^n_{\ep/2}(\la)} g_w(\mu)^{m_\nu(h_w,f_w)}}
\right) \\
& = \lim_{w \to 0} \left( 
\frac{\prod_{\mu \in Z(h_w,g_w) \cap \B D^n_{\ep_0/2}(\la)} f_w(\mu)^{m_\mu(h_w,g_w)}}
{\prod_{\nu \in Z(h_w,f_w) \cap \B D^n_{\ep_0/2}(\la)} g_w(\mu)^{m_\nu(h_w,f_w)}}
\right)
= c_\la^{\ep_0}(h;f,g)
\end{split}
\]
This proves the theorem.
\end{proof}

\begin{remark}\label{r:explocexp}
The value $c_\la(h;f,g) \in \cc^*$ may be computed more explicitly by the formula
\[
\lim_{w \to 0} \left( \frac{ (f(\la) - w)^{m_\la(h,g)}}
{\prod_{\nu \in Z(h) \cap f^{-1}(\{w\}) \cap \B D^n_{\ep/2}(\la)} g(\nu)^{m_\nu(h,f - w)}  } \right)
\]
where $w \in \cc$ approaches zero in Euclidean metric on $\cc$ and where $\ep > 0$ is chosen such that
\[
\B D^n_\ep(\la) \su U \q \T{and} \q \B D^n_\ep(\la) \cap \big( Z(h,g) \cup Z(h,f) \big) = \{\la\}
\]
This is a consequence of the proof of Theorem \ref{t:joystalk} and Remark \ref{r:expjoytor}.
\end{remark}

\begin{remark}
The quantity $c_\la(h;f,g) \in \cc^*$ can be expressed as a limit of a sequence instead of as a limit point of a holomorphic function, see Theorem \ref{t:joystalkint}. To see this it suffices to apply Theorem \ref{t:conjoy} instead of Theorem \ref{th:analyticity} in the proof of Theorem \ref{t:locjoytor} and Theorem \ref{t:joystalk}.
\end{remark}

The next theorem is the main result of this paper. It provides a local formula for the joint torsion.

\begin{thm}\label{t:joytorgen}
Let $A = (A_1,\ldots,A_n)$ be a commuting $n$-tuple on the Hilbert space $\sH$ and let $h : \T{Sp}(A) \to \cc^{n-1}$ and $f,g : \T{Sp}(A) \to \cc$ be holomorphic. Suppose that 
\[
Z(h,f) \cap \T{Sp}_{\T{ess}}(A) = \emptyset = Z(h,g) \cap \T{Sp}_{\T{ess}}(A)
\]
Then
\[
\T{JT}\big( K(h,\sH); f,g \big) = \prod_{\la \in \big( Z(h,f) \cup Z(h,g)\big) \cap \Om} c_\la(h;f,g)^{\T{Ind}(A - \la)}
\]
where $\Om := \T{Sp}(A)^\ci$.
\end{thm}
\begin{proof}
Let $\wit h : U \to \cc^{n-1}$ and $\wit f, \wit g : U \to \cc$ be holomorphic representatives for $h \in \C O(\T{Sp}(A))^{n - 1}$ and $f,g \in \C O(\T{Sp}(A))$ on an open set $U \ssu \T{Sp}(A)$.

Choose an $\ep > 0$ such that
\[
\B D_\ep^n(\la) \su U \, \, , \, \, \, \B D_\ep^n(\la) \cap \big( Z(h,f) \cup Z(h,g) \big) = \{ \la \}
\, \, , \, \, \,
\B D_\ep^n(\la) \cap \T{Sp}_{\T{ess}}(A) = \emptyset
\]
for all $\la \in Z(h,f) \cup Z(h,g)$. We may furthermore arrange that
\[
\B D_\ep^n(\la) \su \Om
\]
for all $\la \in \big( Z(h,f) \cup Z(h,g) \big) \cap \Om$. Finally, we may assume that
\[
\B D_\ep^n(\la) \cap \B D_\ep^n(\mu) = \emptyset
\]
whenever $\la \neq \mu$ and $\la,\mu \in Z(h,f) \cup Z(h,g)$.

Choose a $\de > 0$ such that
\[
\begin{split}
& \de \leq \inf\big\{ |f(z)| \mid z \in Z(h,g) \T{ and } f(z) \neq 0 \big\} \q \T{and} \\
& \de \leq \inf\big\{ |f(z)| \mid z \in \bigcap_{\la \in Z(h,g) \cup Z(h,f)} \big( \cc^n \sem \B D_{\ep/2}^n(\la) \big) \cap Z(h) \big\}
\end{split}
\]

The map $\B B_\de(0) \to \C O(U)^{n + 1}$, $w \mapsto (\wit h, \wit f - w, \wit g)$ is then holomorphic and it satisfies conditions $(1)$, $(2)$, and $(3)$ of Theorem \ref{t:locjoytor}. Furthermore, we have that
\[
\big( Z(h,f - w) \cup Z(h,g) \big) \cap \big( \cup_{\la \in Z(h,g) \cup Z(h,f)} \B D_{\ep/2}^n(\la) \big)
= Z(h,f - w) \cup Z(h,g)
\]
for all $w \in \B B_\de(0)$.

An application of Theorem \ref{t:locjoytor} now yields that
\[
\T{JT}\big( K(h,\sH) ; f,g \big) = \lim_{w \to 0}\left( 
\prod_{\la \in Z(h,g) \cup Z(h,f)} \frac{f_w(\la)^{m_\la(h,g) \cd \T{Ind}(A-\la)} }
{\prod_{\nu \in \B D_{\ep/2}^n(\la) \cap Z(h,f_w) \cap \Om} g(\nu)^{m_\nu(h,f_w) \cd \T{Ind}(A - \nu)}}
\right)
\]
where $f_w := f - w$ for all $w \in \B B_\de(0)$.

Since the Fredholm index is a homotopy invariant and since $\B D_{\ep/2}^n(\la) \cap \T{Sp}_{\T{ess}}(A) = \emptyset$ for all $\la \in Z(h,g) \cup Z(h,f)$ we obtain that
\[
\T{JT}\big( K(h,\sH) ; f,g \big) = \lim_{w \to 0} \prod_{\la \in \big(Z(h,g) \cup Z(h,f) \big) \cap \Om} 
\frac{f_w(\la)^{m_\la(h,g) \cd \T{Ind}(A - \la)}  }
{\prod_{\nu \in \B D_{\ep/2}^n(\la) \cap Z(h,f_w)} g(\nu)^{m_\nu(h,f_w) \cd \T{Ind}(A - \la)}  }
\]

However, from Theorem \ref{t:joystalk} we know that the limit
\[
\lim_{w \to 0} \frac{f_w(\la)^{m_\la(h,g)}}
{\prod_{\nu \in \B D_{\ep/2}^n(\la) \cap Z(h,f_w)} g(\nu)^{m_\nu(h,f_w)}}
\]
exists and agrees with $c_\la(h;f,g)$ for each $\la \in \big( Z(h,g) \cup Z(h,f) \big) \cap \Om$. This proves the present theorem.
\end{proof}

\section{Application: Tame symbols of complex analytic curves}\label{s:tamsymcom}

\subsection{Preliminaries on complex analytic spaces}
Consider an open set $U \su \cc^n$ together with holomorphic functions $h_1,\ldots,h_m : U \to \cc$.

Define the zero-set
\[
Z(h) := \{ z \in U \mid h_1(z) =\ldots h_m(z) = 0 \}.
\]

Let $\C O_U$ denote the sheaf of holomorphic functions on $U$.

For each $z \in Z(h)$, let
\[
\inn{h}_z \su \C O_z
\]
denote the ideal generated by $h_1,\ldots,h_m : U \to \cc$ in the stalk $\C O_z$ of $\C O_U$ at the point $z \in Z(h)$.

The \emph{complex model space} associated to $h_1,\ldots,h_m : U \to \cc$ is the pair
\[
\big( Z(h), \C O/ \inn{h} \big|_{Z(h)} \big)
\]
consisting of the Hausdorff space $Z(h)$ and the restriction of the quotient sheaf $\C O/\inn{h}$ to the zero-set $Z(h)$. Thus, for an open set $V \su \cc^n$, a local section $s \in (\C O/\inn{h})(V \cap Z(h))$ is a collection
\[
\{ s_z \}_{z \in V \cap Z(h)} \q s_z \in \C O_z/\inn{h}_z
\]
such that for each $z_0 \in V \cap Z(h)$ there exists an open set $W \su \cc^n$ with $z_0 \in W$ and a holomorphic map $t : W \to \cc$ with
\[
s_z = [t_z] \q \forall z \in W \cap Z(h)
\]
where $[t_z]$ denotes image of $t \in \C O(W)$ under the map
\[
\C O(W) \to \C O_z \to \C O_z / \inn{h}_z \, .
\]

We recall the following definition from \cite{GrRe:CAS}.

\begin{dfn}
A \emph{complex analytic space} is a pair $(X, \C O_X)$ where $X$ is a Hausdorff space and $\C O_X$ is a sheaf of local $\cc$-algebras on $X$, such that for each $x \in X$ there exist an open neighborhood $V \su X$ and a complex model space $\big( Z(h), \C O/\inn{h}\big|_{Z(h)} \big)$ together with an isomorphism 
\[
(\phi,\ov \phi) : \big(Z(h),\C O/\inn{h}\big|_{Z(h)}\big) \to (V,\C O_X|_V) 
\]
of sheafs of local $\cc$-algebras. Thus, $\phi : Z(h) \to V$ is an isomorphism of topological spaces and $\ov \phi(W) : \C O_X(W) \to \big( \C O/\inn{h} \big)(\phi^{-1}(W))$ is an isomorphism of $\cc$-algebras for each open set $W \su V$.

We will refer to $\big( Z(h), \C O/\inn{h}\big|_{Z(h)} \big)$ as a \emph{local model} for $(X, \C O_X)$ near the point $x \in X$.
\end{dfn}

\subsection{The Tate tame symbol}
Let $(X, \C O_X)$ be a complex analytic space and let us fix a point $x \in X$.
\medskip

\noindent \emph{Suppose that $\T{dim}_x(X) = 1$.}
\medskip

Thus, there exists an open neighborhood $V \su X$ of $x$ and a section $f \in \C O_X(V)$ such  that $f(x) = 0$ and such that the quotient $\C O_{X,x}/\inn{f_x}$ is a finite dimensional vector space over $\cc$, where $\inn{f_x}$ is the ideal generated by $f$ in the stalk $\C O_{X,x}$.

The following assumption will also remain valid throughout this subsection.

\begin{assu}\label{a:locmod}
Suppose that there exists a local model $\big( Z(h), \C O/\inn{h}\big|_{Z(h)} \big)$ for $(X, \C O_X)$ near $x \in X$ where
\[
h = (h_1,\ldots,h_{n - 1}) : U \to \cc^{n - 1}
\]
is holomorphic and $U \su \cc^n$ is open.
\end{assu}

We are now ready to define the Tate tame symbol at the point $x \in X$.

\begin{dfn}\label{d:tatsym}
Let $V \su X$ be an open neighborhood of $x \in X$ and let $f,g \in \C O_X(V)$ with
\[
\C O_{X,x}/\inn{f_x} \q \M{and} \q \C O_{X,x}/\inn{g_x}
\]
of finite dimension over $\cc$. Let $\big( Z(h), \C O/\inn{h}\big|_{Z(h)} \big)$ be a local model for $(X, \C O_X)$ near $x \in X$ as in Assumption \ref{a:locmod} with associated isomorphism
\[
(\phi,\ov \phi) : \big(Z(h),\C O/\inn{h}\big|_{Z(h)}\big) \to (V,\C O_X|_V) 
\]

The \emph{Tate tame symbol} of $f,g \in \C O_X(V)$ at $x \in X$ is defined by
\[
c_x( X;f,g) := c_{\phi^{-1}(x)}\big( h|_{\Om} ; \wit f, \wit g\big)
\]
where $\Om \su U$ is an open neighborhood of $\phi^{-1}(x) \in Z(h)$ and $\wit f, \wit g : \Om \to \cc$ are holomorphic functions such that the identities
\[
[\wit f_\mu] = \ov \phi (V)(f)_\mu \q \M{and} \q [\wit g_\mu] = \ov \phi (V)(g)_\mu
\]
hold in $\C O_\mu/\inn{h}_\mu$ for all $\mu \in \Om \cap Z(h)$.
\end{dfn}

We need to show that that the Tate tame symbol is independent of the various choices. This is part of the next proposition, which also gives a more concrete expression for the tame symbol.

Let us introduce the notation
\[
m_x(f) := \M{dim}_\cc \C O_{X,x}/\inn{f_x}
\]
for any local section $f \in \C O_X(V)$. We remark that $m_x(f) = 0$ whenever $f(x) \neq 0$.

\begin{prop}
Let $V \su X$ be an open neighborhood of $x \in X$ and let $f,g \in \C O_X(V)$ with
\[
\C O_{X,x}/\inn{f_x} \q \M{and} \q \C O_{X,x}/\inn{g_x}
\]
finite dimensional. Then
\[
c_x(X;f,g) = \lim_{w \to 0} \frac{\big( f(x) - w\big)^{m_x(g)}}
{\prod_{y \in \Te \cap f^{-1}(\{w\})} g(y)^{m_y(f - w)}}
\]
where $w \in \cc$ approaches $0$ in the Euclidean metric on $\cc$, and where $\Te \su V$ is any open neighborhood of $x$ such that
\begin{enumerate}
\item $\ov{\Te} \su V$ and $\ov{\Te}$ is compact.
\item $\ov{\Te} \cap \big( Z(f) \cup Z(g) \big) \su \{ x \}$.
\end{enumerate}
In particular, we have that $c_x(X;f,g)$ is well-defined.
\end{prop}
\begin{proof}
Let us choose a local model $\big( Z(h), \C O/\inn{h} \big|_{Z(h)} \big)$ for $(X,\C O_X)$ near the point $x \in X$, and let
\[
(\phi, \ov \phi) : \big( Z(h), \C O/\inn{h} \big|_{Z(h)} \big) \to (V,\C O_X|_V)
\]
denote the associated isomorphism.

Put $\la := \phi^{-1}(x)$ and choose lifts
\[
\wit f, \wit g : \Om \to \cc
\]
of the sections $\ov \phi (V)(f), \ov \phi (V)(g) \in \big( \C O/\inn{h} \big)(Z(h))$ near $\la \in \Om$. We remark that
\[
\wit f(\mu) = f(\phi(\mu)) \q \M{and} \q \wit g(\mu) = g(\phi(\mu))
\]
for all $\mu \in \Om \cap Z(h)$.

Furthermore, we notice that $(\phi, \ov \phi)$ induces isomorphisms
\[
\begin{split}
& \C O_{X,\phi(\mu)}/\inn{f_{\phi(\mu)} - w} \cong \C O_\mu/\inn{h,\wit f - w}_\mu \q \M{and} \\
& \C O_{X,\phi(\mu)} /\inn{g_{\phi(\mu)}} \cong \C O_\mu/\inn{h,\wit g}_\mu
\end{split}
\]
for all $\mu \in \Om \cap Z(h)$ and all $w \in \cc$. In particular, we have the identities
\[
m_\mu(h,\wit f - w) = m_{\phi(\mu)}(f - w) \q \M{and} \q
m_\la(h,\wit g) = m_x(g)
\]
for all $\mu \in \Om \cap Z(h)$ and all $w \in \cc$.

Let us now choose an $\ep > 0$ such that
\[
\big( Z(h, \wit f) \cup Z(h,\wit g) \big) \cap \B D_\ep^n(\la) \su \{\la\}
\q \M{and} \q
\B D_\ep(\la) \su \Om
\]
Thus, by Remark \ref{r:explocexp} and the above observations, we obtain that
\begin{equation}\label{eq:expexp}
\begin{split}
c_\la(h|_{\Om}; \wit f, \wit g) & = \lim_{w \to 0} \frac{(\wit f(\la) - w)^{m_\la(h,\wit g)}}
{\prod_{\mu \in Z(h) \cap \wit f^{-1}(\{w\}) \cap \B D_{\ep/2}^n(\la)} \wit g(\mu)^{m_\mu(h,\wit f - w)}  } \\
& = \lim_{w \to 0} \frac{(f(x) - w)^{m_x(g)}}
{\prod_{\mu \in Z(h) \cap \wit f^{-1}(\{w\}) \cap \B D_{\ep/2}^n(\la)} g(\phi(\mu))^{m_{\phi(\mu)}(f - w)}  } \\
& = \lim_{w \to 0} \frac{(f(x) - w)^{m_x(g)}}{\prod_{y \in \phi\big( Z(h) \cap \B D_{\ep/2}^n(\la) \big) \cap f^{-1}(\{w\}) } g(y)^{m_y(f - w)}  }
\end{split}
\end{equation}

This proves the statement of the proposition for the open neighborhood
\[
\Te := \phi\big( Z(h) \cap \B D_{\ep/2}^n(\la) \big) \su V
\]
of $x \in X$.

To prove the general statement, we now let $\Te \su V$ be an arbitrary open neighborhood of $x \in X$ such that the conditions $(1)$ and $(2)$ are satisfied. Since the limit in \eqref{eq:expexp} is independent of the choice of $\ep > 0$, we may assume that
\[
\phi\big( Z(h) \cap \B D_{\ep/2}^n(\la) \big) \su \Te.
\]
It then suffices to find a $\de > 0$ such that
\[
\Te \cap f^{-1}(\{w\}) \su \phi\big( Z(h) \cap \B D_{\ep/2}^n(\la) \big) \cap f^{-1}(\{w\})
\]
for all $w \in \B B_\de(0)$. But this property is satisfied with
\[
\de := \inf\big\{ |f(y)| \, \big | \, y \in \ov{\Te}\sem \phi\big( Z(h) \cap \B D_{\ep/2}^n(\la)\big) \big\}.
\]
\end{proof}

\begin{prop}
Let $V \su X$ be an open neighborhood of $x \in X$ and let $f_j, t \in \C O_X(V)$, $j = 1,2,3$, be local sections over $V$ such that
\[
\C O_{X,x}/\inn{(f_j)_x} \, \, , \, \, \, \C O_{X,x}/\inn{t_x} \, \, \M{ and } \, \, \,
\C O_{X,x}/\inn{1 - t_x}
\]
are finite dimensional vector spaces over $\cc$. Then the Tate tame symbol satisfies the properties
\begin{enumerate}
\item $c_x(X;f_1,f_2) = c_x(X;f_2,f_1)^{-1}$;
\item $c_x(X;f_1,f_2 f_3) = c_x(X;f_1,f_2) \cd c_x(X;f_1,f_3)$;
\item $c_x(X;t,1-t) = 1$.
\end{enumerate}
\end{prop}
\begin{proof}
This is an easy consequence of the definition of the Tate tame symbol, see Definition \ref{d:tatsym} and Theorem \ref{t:joystalk}.
\end{proof}
For completeness, let us compute the resulting formula in the case when $x\in X$ above is a regular point.
\begin{prop}
Let $X$ be a Riemann surface and let $f, g$ be holomorphic functions defined in a neighbourhood of $x \in X$. Then
\begin{equation}\label{eq:reg}
c_x(X;f,g)=(-1)^{m_x(f)m_x(g)}\lim_{w\rightarrow x}\frac{f(w)^{m_x(g)}}{g(w)^{m_x(f)}}.
\end{equation}

\end{prop}
\begin{proof}
Since the computation is local, we can just as well assume that $x=0\in \cc$ and $f$ and $g$ are two functions holomorphic in a neighbourhood of $0$. Hence they can be written in the form 
\[
f(z)=z^{m_0(f)}\phi (z)\mbox{ and } g(z)=z^{m_0(g)}\psi (z),
\]
where both $\phi (0)\neq 0$ and $\psi (0)\neq 0$. Then, since the formula \eqref{eq:reg} has the properties listed in the proposition above, the computation reduces to checking that
\[
c_0(\cc;z,z)=-1 \, ,\, \, c_0(\cc;\phi,z )=\phi (0) \, , \, \, \mbox{ and } \, \,
c_0(\cc;\phi,\psi) = 1. 
\]
But this is obvious by Remark \ref{r:explocexp}.
\end{proof}

\bibliographystyle{amsalpha-lmp}

\def\cprime{$'$} \def\cprime{$'$}
\providecommand{\bysame}{\leavevmode\hbox to3em{\hrulefill}\thinspace}
\providecommand{\MR}{\relax\ifhmode\unskip\space\fi MR }
\providecommand{\MRhref}[2]{%
  \href{http://www.ams.org/mathscinet-getitem?mr=#1}{#2}
}
\providecommand{\href}[2]{#2}

\end{document}